\documentclass[12pt]{article}
\usepackage{amsfonts,epsfig,latexsym,amscd,amsmath,theorem,mathrsfs}

\newcommand{\Sym}{{\rm Sym}} 
\newcommand{\Jac}{{\rm Jac}}

\newcommand{\R}{{\mathbb{R}}}
\renewcommand{\P}{{\mathbb{P}}}

\newcommand{\I}{{\mathbb{I}}}

\newcommand{\CP}{{\mathbb{C}}{{P}}}

\newcommand{\beq}{\begin{equation}}
\newcommand{\eeq}{\end{equation}}
\newcommand{\bea}{\begin{eqnarray}}
\newcommand{\eea}{\end{eqnarray}}
\newcommand{\ben}{\begin{eqnarray*}}
\newcommand{\een}{\end{eqnarray*}}
\newcommand{\bem}{\begin{enumerate}}
\newcommand{\eem}{\end{enumerate}}
\newcommand{\ii}{\item}
\newcommand{\ra}{\rightarrow}

\newcommand{\cd}{\partial}

\newcommand{\wh}{\widehat}
\newcommand{\less}{\backslash}

\newcommand{\X}{{\mathscr{X}}}

\def \d{\mathrm{d}}
\newcommand{\dstar}{\delta}
\newcommand{\ip}[1]{\langle #1 \rangle}
\newcommand{\ignore}[1]{}

\newcommand{\vol}{{\rm vol}}

\newcommand{\ol}{\overline}

\newcommand{\eps}{\varepsilon}
\renewcommand{\phi}{\varphi}

\theoremstyle{plain}
\newtheorem{thm}{Theorem}
\newtheorem{lemma}[thm]{Lemma}
\newtheorem{prop}[thm]{Proposition}
\newtheorem{cor}[thm]{Corollary}

{\theorembodyfont{\rmfamily}

}

\newcommand{\news}{\setcounter{equation}{0}}
\newenvironment{proof}{\noindent{\it Proof:\, }}{\hfill$\Box$\vspace*{0.5cm}
}
\newenvironment{lproof}[1]{\noindent{\it Proof #1:\, }}{\hfill$\Box$\vspace*{0.5cm}

}

\begin{document}

\title{The geometry of the space of vortices on a two-sphere in the Bradlow limit}
\author{R.I. Garc\'ia Lara\thanks{E-mail: {\tt rene.garcia@correo.uady.mx}}\\ 
Facultad de Matematicas, Universidad Autonoma de Yucatan\\
Merida, Mexico \\ \\
J.M. Speight\thanks{E-mail: {\tt j.m.speight@leeds.ac.uk} (corresponding author)}\\
School of Mathematics, University of Leeds\\
Leeds LS2 9JT, England}

\maketitle

\begin{abstract}
It is proved that the normalized $L^2$ metric on the moduli space
of $n$-vortices on a two-sphere, endowed with any Riemannian metric, converges uniformly in the Bradlow limit to
the Fubini-Study metric. This establishes, in a rigorous setting, a longstanding informal conjecture of Baptista and Manton. 
\end{abstract}

\maketitle

\section{Introduction}
\news

Vortices are the simplest class of topological solitons arising in gauge theory, making them interesting objects of mathematical study, independent of their phenomenological applications in condensed matter physics and cosmology. At critical coupling, they satisfy a self-duality (or Bogomol'nyi) type condition which implies that static vortices exert no net force on one another. The moduli space $M_n$ of static $n$-vortex solutions is thus exceptionally large, forming a complex $n$-manifold. There is a well-developed programme for studying their low energy dynamics in this regime, comprehending classical, quantum and
statistical mechanics, originally proposed by Manton. The key object underpinning this programme is a natural Riemannian metric $g$ on $M_n$, called the $L^2$ metric, obtained by restricting the model's kinetic energy functional to
$TM_n$. Vortex dynamics is modelled by geodesic motion in $(M_n,g)$.  For
a comprehensive review of the geodesic approximation to vortex dynamics, see
\cite[ch.\ 3]{mansut}.

The $L^2$ metric on $M_n$ is, then, the object of strong and sustained mathematical interest. There are very few situations in which $g$ can be computed
exactly, however. Other than the rather trivial case of $M_1$ when physical space  is homogeneous, Strachan has exactly computed the metric on $M_2$ on the hyperbolic plane \cite{str}. Samols \cite{sam} adapted this calculation to obtain a semi-explicit localization formula for the metric on an arbitrary surface (for example  the Euclidean plane) which is useful for numerics, and can be used to prove many interesting global and qualitative properties of $g$, but it does not yield explicit formulae. Even the metric on $M_1$ is nontrivial if physical space is not homogeneous \cite{man-ricciflow}.

This paper examines one case where explicit progress is possible, at least in a certain parametric limit: vortices on a two-sphere. Bradlow \cite{bra}
observed that, on a compact Riemann surface $\Sigma$, there is a lower bound, proportional to $n$, on the area $|\Sigma|$ required for the surface to accommodate $n$-vortices. Above this bound, $M_n$ is biholomorphic to $\Sym_n\Sigma$, the $n$-fold symmetric product of $\Sigma$, since vortices are uniquely determined by the collection of $n$ points on $\Sigma$ at which their Higgs field vanishes. As $|\Sigma|$ approaches the bound from above, the vortices spread out and lose their spatial localization. Precisely at the bound, they delocalize entirely: the Higgs field vanishes identically and the magnetic field is uniform. Such vortex solutions are sometimes called ``dissolved" vortices, and the limit in which they appear called the ``dissolving" limit. The moduli space of dissolved vortices is precisely the moduli space of constant curvature connexions on the line bundle of which the Higgs field is a section. If the genus $r$ of $\Sigma$ is positive this is a torus of dimension $2r$ which may naturally be identified with the Jacobian variety of $\Sigma$. It also has an $L^2$ metric, to which the $L^2$ metric $g$ on $M_n$ conjecturally degenerates (more precisely, $g$ conjecturally degenerates to the pullback of the
 metric on the torus by the Abel-Jacobi map $\Sym_n\Sigma\ra\Jac(\Sigma)$), a scenario studied in detail in \cite{manrom}. If $r=0$, that is, $\Sigma=S^2$, the moduli space of dissolved vortices is a point, while $M_n\equiv \CP^n$. To understand the metric $g$ in the dissolving limit we must rescale it by (for example) demanding that $(M_n,g)$ has some normalized volume. In this context, Baptista and Manton \cite{bapman} made the remarkable conjecture (on the {\em round} two-sphere), that $g$ converges in the dissolving limit to the Fubini-Study metric on $\CP^n$, implying an enormous gain in symmetry in this limit. 

The purpose of this paper is to give a precise formulation and rigorous proof of a slight generalization of Baptista and Manton's conjecture: we will show that, in the dissolving limit, the normalized $L^2$ metric on $M_n(S^2)$ converges uniformly to the Fubini-Study metric, regardless of the metric on $S^2$ (that is, we remove the assumption that the domain sphere is round). The $L^2$ geometry in the dissolving limit is thus unreasonably simple: a family of metrics with  (generically) no nontrivial isometries at all converges uniformly to a homogeneous metric of constant holomorphic sectional curvature. The convergence established implies \cite{banura} that the spectrum of the Laplacian on $M_n$ converges uniformly to the spectrum of the Fubini-Study Laplacian, which is easily computed \cite{bergaumaz}. The theorem thus has immediate consequences for the
{\em quantum} dynamics of vortices on $S^2$ in the dissolving limit.

 The paper is structured as follows. In section \ref{sec2} we formulate the model and give a precise statement of the theorem. In section \ref{sec3} we review the notion of {\em pseudo-vortices}, introduced by Baptista and Manton, and prove that vortices converge uniformly to pseudo-vortices in the dissolving limit. In section \ref{sec4}, we establish convergence of the metric, finishing the proof of the main theorem, while in section \ref{sec5} we study the convergence of the spectrum of the Laplacian.

\section{Statement of the theorem}\label{sec2}\news

Let $(\Sigma,g_\Sigma)$ be an oriented Riemannian manifold diffeomorphic to $S^2$ and $(L,h)$ be a hermitian line bundle over $\Sigma$ of degree $n\geq 0$. To any section $\phi$ of $L$ and
unitary connexion $A$ on $L$ we associate the energy
\beq
E(\phi,A)=\int_\Sigma\left(\frac12|\d_A\phi|^2+\frac12|F_A|^2
+\frac18(\tau-|\phi|^2)^2\right)
\eeq
where $\tau>0$ is a constant parameter and $F_A$ is the curvature of $A$. If we choose a local section $\eta$ of $L$ with $|\eta|^2=h(\eta,\eta)=1$, we may identify the connexion $A$ locally with the real one-form $h(i\d_A\eta,\eta)$ which, in a slight abuse of notation, we also denote $A$. Then for a general local section $f\eta$, where $f$ is smooth and complex valued, 
\beq
\d_A(f\eta)=(\d f -iAf)\eta.
\eeq
We adopt the convention that the curvature $F_A$ is real, coinciding (globally) with $\d A$. The energy functional is invariant under gauge transformations,
\beq
(\phi,A) \mapsto (e^{i\chi}\phi,A+\d\chi),
\eeq
where $\chi:\Sigma\ra\R$ is smooth.

It is well known \cite{bog,bra} that $E(\phi,A)\geq\tau\pi n$ with equality if and only if
\bea
\label{v1}
\ol\cd_A\phi&=&0\\
\label{v2}
*F_A&=&\frac12(\tau-|\phi|^2),
\eea
where $\ol\cd_A$ denotes the $0,1$ part of $\d_A$ (with respect to the complex structure on $\Sigma$ defined by its orientation and metric) and $*$ is the Hodge isomorphism defined by $g_\Sigma$.
Solutions of this system of equations are called vortices. Since they attain a topological lower bound on $E$, they are automatically
global minima of $E$. 

A necessary condition for existence of vortices was obtained by Bradlow \cite{bra} by integrating \eqref{v2} over $\Sigma$:
\beq\label{bradconstraint}
2\pi n = \frac12\tau|\Sigma|-\frac12\int_\Sigma |\phi|^2\leq
\frac12\tau|\Sigma|
\eeq
where $|\Sigma|$ denotes the area of $\Sigma$. Hence, if vortices exist,
\beq
\eps:=\tau|\Sigma|-4\pi n\geq 0.
\eeq
We refer to $\eps\searrow 0$ as the Bradlow limit. More precisely, we shall consider the limit where the metric $g_\Sigma$ (and hence $|\Sigma|$) is fixed, and
$\tau\searrow 4\pi n/|\Sigma|$. This is more 
convenient for our purposes than the limit considered by
Baptista and Manton ($\tau\equiv 1$ and $g_\Sigma$ varying through round metrics, with $|\Sigma|$
approaching $4\pi n$ from above) because it allows us to define Sobolev spaces with fixed norms determined by $g_\Sigma$. That our limit includes theirs follows from the observation that, if $(\phi,A)$ solves the vortex equations on $(\Sigma, g_\Sigma)$ at coupling $\tau=\tau_0$, then $(\tau_0^{-1/2}\phi,A)$ satisfies the vortex equations on $(\Sigma, \tau_0g_\Sigma)$ at coupling $\tau=1$. Note that
at $\eps=0$ any vortex must have $\phi=0$ and $*F_A$ constant, so that vortices completely lose their spatial localization, which is why this is often called the {\em dissolving} limit. For $\eps>0$, every vortex has $\|\phi\|_{L^2}^2=\eps$. 

The moduli space of vortices $M_n$ is the space of gauge equivalence classes of solutions of \eqref{v1}, \eqref{v2}. This is a complex $n$-manifold canonically diffeomorphic to $\Sym_n\Sigma$, the $n$-fold symmetric product of $\Sigma$, the canonical diffeomorphism $F:M_n\ra\Sym_n\Sigma$ being the map which sends the (gauge equivalence class of a) vortex solution $(\phi,A)$ to the set of points at which $\phi$ vanishes, counted with multiplicity. That is, we map the gauge equivalence class  of solutions $[(\phi,A)]$ to the degree $n$ effective divisor
$F([(\phi,A)])=(\phi)$ in $\Sigma$. In this way we identify the $\eps$-dependent family of manifolds $M_n$ with the single fixed manifold
$\Sym_n\Sigma$ which, in a slight abuse of notation, we will also denote $M_n$.

Since $\Sigma=S^2$, we may 
specify a degree $n$ effective divisor $D$ by giving a polynomial
\beq
a_0+a_1z+\cdots+a_nz^n
\eeq
vanishing on $D$, where $z$ is a stereographic coordinate. Two such polynomials define the same divisor if and only if one is a constant multiple of the other, so we identify
$(\phi,A)$, up to gauge, with $[a_0,a_1,\ldots,a_n]\in\CP^n$. Hence
$M_n\equiv \CP^n$. 

For our purposes, an alternative identification $M_n\equiv \CP^n$ is more convenient, however. Let us choose and fix, once and for all, a unitary connexion $\wh{A}$ on $L$ of constant curvature. This is unique up to gauge.  We equip $L$ with the holomorphic structure whose Dolbeault operator $\ol\cd_L=\ol\cd_{\wh{A}}$. Then 
$H^0(L)$, the space of holomorphic sections of $L$, is a complex vector space of dimension $n+1$. Each nonzero section in $H^0(L)$ defines an effective divisor of degree $n$ (on which it vanishes), and two sections define the same divisor if and only if one is a constant nonzero multiple of the other. Hence we have an identification of $M_n$ with the complex projective space of lines through $0$ in
$H^0(L)$. In other words, we identify $M_n$ with $\CP^n$ via the bijective map
\beq
f:M_n\ra\P(H^0(L)),\qquad [(\phi,A)]\mapsto\{\psi\in H^0(L): (\psi)=(\phi)\}
\eeq
which identifies $[(\phi,A)]$ with the line in $H^0(L)$ consisting of holomorphic sections vanishing on the same divisor as $\phi$. This allows us to equip $M_n$ with a Fubini-Study metric as follows. The $L^2$ inner product on $\Gamma(L)$ defines  a hermitian inner product on $H^0(L)$. Denote by $S$ the unit sphere in $(H^0(L),\|\cdot\|_{L^2})$, and by
$\pi:S\ra\P(H^0(L))$, $\pi(\psi)=[\psi]$, the Hopf fibration. Then the Fubini-Study metric of constant holomorphic sectional curvature $2$ is precisely the metric $g_{FS}$ on $\P(H^0(L))$ with respect to which $\pi$ is a Riemannian submersion. By the {\em Fubini-Study metric} on $M_n$ we will always mean
\beq
g_0:=f^*g_{FS}.
\eeq

A priori, there is no reason to expect $g_0$ to have any relevance to vortex dynamics, which is controlled \cite[ch.\ 3]{mansut} by a natural Riemannian metric $g$ on $M_n$ called the $L^2$ metric. This is
defined as follows. Any smooth curve $(\phi(t),A(t))$ of vortex solutions through $(\phi(0),A(0))$ tautologically defines a tangent vector to $M_n$ at $[(\phi(0),A(0))]$, to which we associate the
squared length
\beq
\int_\Sigma\left(|P(\dot\phi(0),\dot{A}(0))|^2\right)
\eeq
where $P$ denotes projection $L^2$ orthogonal to the gauge orbit through
$(\phi(0),A(0))$. No explicit formula for the $L^2$ metric is known (except in the trivial case that $n=1$ and $\Sigma$ is round), but it is known to be 
k\"ahler, and a formula for its k\"ahler class was obtained by Baptista \cite{bap-l2}. From this, it follows that the volume of $M_n$ is
\beq
|M_n|=\frac{\pi^n}{n!}(\tau|\Sigma|-4\pi n)^n=\frac{\pi^n\eps^n}{n!},
\eeq
which vanishes in the Bradlow limit. Hence $g$ degenerates in this limit so, to discuss its convergence, we must rescale. We define the
normalized $L^2$ metric on $M_n$ to be
\beq
g_\eps:=\frac{1}{\eps} g.
\eeq
This is a one-parameter family of metrics on $M_n$ of fixed volume $\pi^n/n!$. 
The purpose of this paper is to establish

\begin{thm}\label{main} In the limit $\eps\searrow 0$, the normalized $L^2$ metric ${g}_\eps$ converges in $C^0$ to ${g}_0$, the Fubini-Study metric on
$M_n$.
\end{thm}

A result of this kind was conjectured by Baptista and
Manton in the case that $\Sigma$ is round, motivated by the approximation
of vortices by what they called {\em pseudo-vortices} \cite{bapman}. These are pairs $(\phi,\wh{A})$ consisting of a connexion of constant curvature, and a section holomorphic with respect to $\ol\cd_{\wh{A}}$ satisfying a normalization condition. Our proof proceeds by generalizing this approximation to arbitrary spheres, and rigorously controlling its errors.

\section{Convergence of vortices to pseudo-vortices}\label{sec3}
\news

Since we are concerned with the limit $\eps\searrow 0$, we assume henceforth that 
\beq
\frac{4\pi n}{|\Sigma|}<\tau<\frac{1+4\pi n}{|\Sigma|},
\eeq
so $\eps\in(0,1)$. Recall we have fixed a connexion $\wh{A}$ on $L$ of constant curvature $*F_{\wh{A}}=2\pi n/|\Sigma|$, and used this to equip $L$ with a holomorphic structure.  Given a divisor $D\in \Sym_n\Sigma$, let us denote by $\wh{\phi}_D\in H^0(L)$ a holomorphic section of $L$ with
$\|\wh\phi_D\|_{L^2}=1$ and $\wh{\phi}^{-1}(0)=D$. This is unique up to multiplication by a
constant in $U(1)$. Let us define the {\em pseudo-vortex} corresponding to the divisor $D$ to be the  pair
$(\sqrt{\eps}\wh\phi_D,\wh{A})$, unique up to $\wh\phi_D\mapsto
e^{ic}\wh\phi_D$ where $c\in\R$ is constant. This pair satisfies the first vortex equation
\eqref{v1}, but not the second \eqref{v2}. It does satisfy \eqref{v2}
``on average," however, that is, it satisfies the constraint
 \eqref{bradconstraint} obtained by
integrating \eqref{v2} over $\Sigma$. This observation led Baptista and Manton to conjecture that, for small $\eps$, the actual vortex
with divisor $D$ should be well approximated by its corresponding
pseudo-vortex \cite{bapman}. In this section we will prove that this
intuition is correct in the following precise sense:

\begin{prop}\label{vortexprop} There exists a constant $C>0$, depending only on $(L,h)$ and
$g_\Sigma$, such that, for all vortex solutions $[(\phi,A)]\in M_n$,
$$
\|\eps^{-1/2}|\phi|-|\wh\phi_D|\|_{C^0}\leq C\eps,\qquad
\|F_{A}-F_{\wh{A}}\|_{C^0}\leq C\eps,
$$
where $D=(\phi)$ and $\|\cdot\|_{C^0}$ denotes the $\sup$ norm. 
\end{prop}

The proof relies on the following observation. Choose and fix $D$, some effective divisor of degree $n$ on $\Sigma$. Then any section $\phi$ of $L$ with $(\phi)=D$ has a unique representative in its gauge equivalence class of the form
\beq
\phi=\sqrt\eps e^{u/2}\wh{\phi}_D
\eeq
where $u:\Sigma\ra\R$ is a smooth function. If we define the connexion $A$ to be
\beq
A=\wh{A}-\frac12*\d u,
\eeq
then the pair $(\phi,A)$ satisfies the first vortex equation \eqref{v1} \cite{flospe}. The connexion $A$ has curvature
\beq\label{gra}
*F_A=*F_{\wh{A}}-\frac12 *\d*\d u
=\frac{2\pi n}{|\Sigma|}+\frac12\Delta u.
\eeq
Hence, $(\phi,A)$ satisfies the second vortex equation \eqref{v2} if and only if
\beq\label{hmv}
\Delta u-\frac{\eps}{|\Sigma|}+\eps|\wh\phi_D|^2e^u
=0.
\eeq
It follows from the main theorem of \cite{bra}, or by direct appeal to \cite{kazwar}, that the solution to \eqref{hmv} exists and is unique and smooth. In this way we reduce the problem of constructing the vortex with divisor $D$ to a semilinear elliptic PDE for $u$. 

The error in the approximation 
\beq
(\phi,A)=(\sqrt{\eps}e^{u/2}\wh\phi_D,\wh{A}-\frac12*\d u)
\approx (\sqrt{\eps}\wh\phi_D,\wh{A})
\eeq
 is controlled by $u$, and hence we seek bounds on this function.
We will use four different norms on $C^\infty(\Sigma)$, namely,
\bea
\|f\|^2_{L^2}&:=&\int_\Sigma |f|^2\nonumber\\
\|f\|^2_{H^1}&:=&\int_\Sigma (|f|^2+|\d f|^2)\nonumber\\
\|f\|_{H^2}^2&:=&\int_\Sigma (|f|^2+|\d f|^2+|\nabla\d f|^2)\nonumber \\
\|f\|_{C^0}&:=&\sup\{|f(x)|:x\in \Sigma\}
\eea
and will denote by $\ip{\cdot,\cdot}_{L^2}$ the $L^2$ inner product.
In the defintion of the $H^2$ norm, $\nabla$ denotes the Levi-Civita connexion associated to $g_\Sigma$. We will also make frequent use of two standard analytic facts which, for convenience, we collect here \cite{donkro,aub}:

\begin{prop}\label{sobprop} There exists a constant $c>0$, depending only on $g_\Sigma$, such that
\bem
\ii For all $f\in C^\infty(\Sigma)$, $\|f\|_{C^0}\leq c\|f\|_{H^2}$ (Sobolev embedding);
\ii For all $f\in C^\infty(\Sigma)$ with $\int_\Sigma f=0$,
$\|f\|_{H^2}\leq c\|\Delta f\|_{L^2}$ (standard elliptic estimate for $\Delta$).
\eem
\end{prop}

Proposition \ref{vortexprop} will quickly follow once we have shown that $\|u\|_{C^0}\leq C\eps$ uniformly in $D$. We begin by bounding the average of $u$ in terms of its zero-average component.
 
\begin{lemma}\label{lem0} Let $u$ be a solution of \eqref{hmv} and $u_0=u-\ol{u}$, where 
$$
\ol{u}:=\frac{1}{|\Sigma|}\int_\Sigma u.
$$
Then $|\ol{u}|\leq\|u_0\|_{C^0}$.
\end{lemma}

\begin{proof} Substituting $u=\ol{u}+u_0$ into \eqref{hmv} and integrating over $\Sigma$ we find (by the Divergence Theorem)
\beq
0-\eps+\eps\int_\Sigma |\wh\phi_D|^2e^{\ol{u}}e^u_0=0,
\eeq
which may be solved for $\ol{u}$,
\beq
\ol{u}=-\log\left(\int_{\Sigma}|\wh{\phi}_D|^2e^{u_0}\right).
\eeq
Now $-\|u_0\|_{C^0}\leq u_0\leq \|u_0\|_{C^0}$ and
$\exp$ is increasing, so
\beq
e^{-\|u_0\|_{C^0}}=e^{-\|u_0\|_{C^0}}\int_{\Sigma}|\wh{\phi}_D|^2\leq
\int_{\Sigma}|\wh{\phi}_D|^2e^{u_0}\leq 
e^{\|u_0\|_{C^0}}\int_{\Sigma}|\wh{\phi}_D|^2=e^{\|u\|_{C^0}},
\eeq
where we have used the fact that $\|\wh{\phi}_D\|_{L^2}=1$. But $\log$ is also increasing, so
\beq
-\|u_0\|_{C^0}\leq \log\left(\int_{\Sigma}|\wh{\phi}_D|^2e^{u_0}\right)
\leq \|u_0\|_{C^0}
\eeq
whence the claim immediately follows. \end{proof}

This observation quickly yields a crude bound on 
$\|u\|_{C^0}$.
 
\begin{lemma}\label{lem1} There exists $C>0$, depending only on $n$ and $g_\Sigma$, such that, for all $D$, the solution $u$ of \eqref{hmv} satisfies 
$\|u\|_{C^0}\leq C$.
\end{lemma}

\begin{proof}
As before, let $u=\ol{u}+u_0$ with $\ol{u}=|\Sigma|^{-1}\int_\Sigma u$. 
Now, by \eqref{gra},
\bea
\frac14\|\Delta u_0\|_{L^2}^2&=&\|*F_A-2\pi n|\Sigma|^{-1}\|_{L^2}^2\nonumber \\
&=& \|F_A\|_{L^2}^2-\frac{(2\pi n)^2}{|\Sigma|}
\nonumber \\
&<& \|F_A\|_{L^2}^2
\nonumber \\
&\leq& 2E(\phi,A) \nonumber \\
&= & 2\pi\tau n\leq 2\pi n(1+4\pi n)/|\Sigma|.
\eea
Hence, $\|\Delta u_0\|_{L^2}\leq C_1$, some constant depending only on $n$ and $g_{\Sigma}$. 
Now $\int_\Sigma u_0=0$ so, by the standard elliptic estimate for $\Delta$ (Proposition \ref{sobprop}), 
\beq\label{dor}
\|u_0\|_{H^2}\leq c\|\Delta u_0\|_{L^2}\leq cC_1,
\eeq
and hence, by the Sobolev embedding (Proposition \ref{sobprop}),
\beq
\|u_0\|_{C^0}\leq c^2C_1.
\eeq
Hence, by Lemma \ref{lem0},
\beq
\|u\|_{C^0}\leq |\ol{u}|+\|u_0\|_{C^0}
\leq 2\|u_0\|_{C^0}\leq 2c^2C_1.
\eeq
\end{proof} 

Equation \eqref{hmv} allows us to improve this crude bound.

\begin{lemma}\label{lem2} There exists $C>0$, depending only on $(L,h)$ and $g_\Sigma$, such that, for all $D$, the solution $u$ of \eqref{hmv} satisfies 
$\|u\|_{C^0}\leq C\eps$.
\end{lemma}

\begin{proof}
Once again, by Proposition \ref{sobprop} and Lemma \ref{lem0}, it suffices to show that
\beq
\|\Delta u_0\|_{L^2}\leq C\eps,
\eeq
where $u_0$ is the zero-average part of $u$. It is convenient to define
\beq\label{alphadef}
\alpha:=\sup\{|\wh{\phi}_D(p)|:(p,D)\in\Sigma\times\Sym_n\Sigma\},
\eeq
noting that this number certainly exists, by continuity and compactness of $\Sigma$, and is, by definition, dependent only on $(L,h)$ and $g_\Sigma$.
Now, by equation \eqref{hmv},
\beq
\|\Delta u_0\|_{L^2}\leq \frac{\eps}{|\Sigma|^{1/2}}
+\eps\alpha^2e^{\|u\|_{C^0}}|\Sigma|^{1/2}\leq C\epsilon,
\eeq
by Lemma \ref{lem1}.
\end{proof}

\begin{lproof}{of Proposition \ref{vortexprop}}
By definition of $u$,
\beq
\frac{\phi}{\sqrt{\eps}}-\wh{\phi}_D
=\wh{\phi}_D(1-e^{u/2}).
\eeq
By the mean value theorem, for each $x\in\Sigma$, there exists
$q(x)$ between $u(x)$ and $0$ such that
\beq
e^{u(x)/2}-e^{0/2}=\frac12 e^{q(x)/2}(u(x)-0).
\eeq
Hence
\beq
\|1-e^{u/2}\|_{C^0}\leq \frac12 e^{\|u\|_{C^0}/2}\|u\|_{C^0},
\eeq
and so,
\beq
\left\|\frac{\phi}{\sqrt{\eps}}-\wh{\phi}_D\right\|_{C^0}
\leq \frac12 \alpha\|u\|_{C^0}e^{\|u\|_{C^0}/2},
\eeq
where $\alpha$ is the constant defined in \eqref{alphadef}. Similarly, $*(F_A-F_{\wh{A}})=\Delta u/2$, so by \eqref{hmv},
\beq
\|F_A-F_{\wh{A}}\|_{C^0}\leq
\frac\eps2\left(\frac{1}{|\Sigma|}+\alpha^2 e^{\|u\|_{C^0}}\right).
\eeq
The claimed inequalities now follow immediately from Lemma \ref{lem2}.
\end{lproof}

\section{Convergence of the metric}\label{sec4}
\news

Consider a smooth curve of solutions of the vortex equations,
$(\phi(t),A(t))$, and denote by $v$ the tangent vector to $M_n$ at $[(\phi(0),A(0))]$ that it generates. Given any such curve, there exists a smooth curve $\wh\phi(t)$ in the unit sphere $S\subset (H^0(L),
\|\cdot\|_{L^2})$ such that
\bea\label{jg}
\phi(t)&=&\sqrt{\eps}\wh\phi(t)e^{\frac12 u(t)+i\chi(t)}\\
\label{joga}
A(t)&=&\wh{A}+\frac12*d u(t)+\d\chi(t)
\eea
where $u(t)$, $\chi(t)$ are smooth curves in 
$C^\infty(\Sigma)$ and, at each time $t$, $u(t)$ satisfies the PDE \eqref{hmv}. We may choose $\chi(t)$ freely: changing it merely gauge transforms each solution $(\phi(t),A(t))$.
Note that $\|\wh\phi(t)\|_{L^2}^2\equiv 1$, so
\beq\label{hopf1}
\ip{\wh{\phi}(t),\dot{\wh{\phi}}(t)}_{L^2}\equiv 0,
\eeq
where here and henceforth an overdot denotes differentiation with respect to $t$.
Without loss of generality, we may also assume that
\beq\label{hopf2}
\ip{i\wh\phi(t),\dot{\wh{\phi}}(t)}_{L^2}\equiv 0.
\eeq
If not, we gauge transform
\beq
\wh{\phi}(t)\mapsto e^{ic(t)}\wh{\phi}(t)
\eeq
by a suitable curve of space-independent functions $c(t)$ and redefine 
$\chi(t)\mapsto \chi(t)-c(t)$. This is convenient because $\wh{\phi}(t)$ then moves orthogonal to the fibres of the Hopf fibration, so by definition of the Fubini-Study metric,
\beq
g_0(v,v)=g_{FS}([\dot{\wh{\phi}}(0)],[\dot{\wh{\phi}}(0)])=\|\dot{\wh{\phi}}(0)\|_{L^2}^2.
\eeq

This, then, is the squared length assigned to $v$ by the Fubini-Study metric $g_0$. We wish to compare this with the squared length of the same vector with
respect to the $L^2$ metric. For this purpose, it is convenient to choose $\chi(t)$ so that $(\dot\phi(0),\dot{A}(0))$ is $L^2$ orthogonal to the gauge orbit through $(\phi(0),A(0))$, and hence $P(\dot\phi(0),\dot{A}(0))=(\dot\phi(0),\dot{A}(0))$. 
We may take $\chi(0)=0$. A general gauge transform is
\beq
(\phi,A)\mapsto (e^{is}\phi,A+\d s),\qquad s\in C^\infty(\Sigma),
\eeq
so a general tangent vector to the gauge orbit through $(\phi,A)$ takes the form
\beq
v=(is\phi,\d s)\in\Gamma(L)\oplus\Omega^1(\Sigma),\qquad s\in C^\infty(\Sigma).
\eeq
The tangent vector to our curve is
\beq
(\dot\phi(0),\dot A(0))=
(
\sqrt\eps(
\dot{\wh\phi}(0)+\wh\phi(0)(
\frac12\dot u(0)+i\dot\chi(0))
)e^{u(0)/2},
\frac12*\d \dot u(0)+\d\dot\chi(0)).
\eeq
Hence we require that, for all $s\in C^\infty(\Sigma)$,
\bea
0&=&\ip{
is\sqrt\eps \wh\phi(0)e^{u(0)/2},
\sqrt\eps(\dot{\wh\phi}(0)+\wh\phi(0)(\frac12\dot u(0)
+i\dot\chi(0))e^{u(0)/2})}_{L^2}
+\ip{\d s,\frac12*\d \dot u(0)+\d\dot\chi(0)}_{L^2}\nonumber \\
&=&\int_\Sigma s\eps e^{u(0)}h(i\wh\phi(0),\dot{\wh\phi}(0)+\frac12 \dot{u}(0)\wh\phi(0)+i\dot\chi(0)\wh\phi(0))+\ip{s,\frac12\dstar*\d\dot{u}(0)+\dstar\d\dot\chi(0)}_{L^2}\nonumber \\
&=&\ip{s,\Delta\dot\chi(0)+\eps|\wh\phi(0)|^2\dot\chi(0)+\eps e^{u(0)}h(i\wh\phi(0),\dot{\wh\phi}(0))}_{L^2}.
\eea
Thus, the gauge orthogonality condition is
\beq\label{gperp}
\Delta\dot\chi(0)+\eps e^{u(0)}|\wh{\phi}(0)|^2\dot\chi(0)=-\eps e^{u(0)}
h(i\wh\phi(0),\dot{\wh{\phi}}(0)).
\eeq
Since $u(t)$ satisfies \eqref{hmv} at each $t$, we find, by differentiating with respect to $t$, that
\beq\label{hmvlin}
\Delta \dot{u}(0)+\eps e^{u(0)}|\wh{\phi}(0)|^2\dot u(0)=-2\eps e^{u(0)}
h(\wh\phi(0),\dot{\wh{\phi}}(0)).
\eeq
The conditions on our curve of solutions imply that $\dot\chi(0)$ and $\dot  u(0)$ satisfy the driven linear PDEs \eqref{gperp} and \eqref{hmvlin}, and, conversely, given $(\wh\phi(0),\dot{\wh{\phi}}(0))$, these PDEs uniquely determine $\dot\chi(0)$ and $\dot u(0)$.
From now on, we restrict attention to the time $t=0$ and drop the parameter $t$ from our notation.

The squared length of the velocity vector with respect to the $L^2$ metric
is
\bea
g(v,v)&=&\|\dot{\phi}\|_{L^2}^2
+\|\dot{A}\|_{L^2}^2\nonumber \\
&=&\eps\int_\Sigma e^{u}\left\{ |\dot{\wh\phi}|^2+h(\dot{\wh\phi},\wh\phi)\dot{u}+2h(\dot{\wh\phi},i\wh\phi)\dot\chi+
|\wh\phi|^2\left(\frac{\dot{u}^2}{4}+\dot\chi^2\right)\right\}\nonumber \\ \label{y-cl}
&& +\frac14\|\d\dot{u}\|_{L^2}^2+\|\d\dot\chi\|_{L^2}^2.
\eea
  Our aim is to show that the normalized $L^2$ metric $g_\eps=\eps^{-1}g$ satisfies a bound of the form
 \beq
|g_\eps(v,v)-\|\dot{\wh{\phi}}\|_{L^2}^2|\leq C\eps\|\dot{\wh{\phi}}\|_{L^2}^2,
\eeq
where $C$ is independent of $v$. Comparing with \eqref{y-cl} it is clear that we must control suitable norms of $\dot u$ and $\dot\chi$.

Note that both $v=\dot\chi$ and $v=\dot u/2$ satisfy a PDE of the form
\beq\label{vpde}
\Delta v+ av=b
\eeq
where $a,b\in C^\infty(\Sigma)$, 
\beq\label{abdef}
a=\eps|\wh\phi|^2e^u,\qquad
b=\left\{\begin{array}{cc}
-\eps e^uh(i\wh\phi,\dot{\wh{\phi}}), & v=\dot\chi,\\
-\eps e^uh(\wh\phi,\dot{\wh{\phi}}), & v=\dot u/2.\end{array}\right.
\eeq
An obvious strategy at this point would be to use the standard elliptic estimate for the operator $\Delta +a$ to bound $\|v\|_{H^2}$ in terms of $\|b\|_{L^2}$. Indeed, since $a\geq 0$ and vanishes only on a finite set, it is clear that the kernel of $ \Delta+a$ is trivial,
so a bound of the form $\|v\|_{H^2}\leq C(a)\|b\|_{L^2}$ is immediately available. The problem with this approach is that the dependence of the coercivity constant $C(a)$ on the function $a$ (hence on $\wh\phi$) is unknown, so it does not yield a uniform bound as we require. 

To proceed further, we appeal to an estimate obtained from the Lax-Milgram Lemma which gives a bound on $\|v\|_{H^1}$ (not $\|v\|_{H^2}$) in terms of $a$ and $b$. This estimate is likely to be useful in a wide variety of contexts, so we state and prove it in some generality.
At first sight, the bound looks (perhaps needlessly) elaborate, and one is tempted to replace the right hand side by a simpler but less sharp expression. Our argument will require the full detail of the bound as stated, however.

\begin{lemma} \label{LMlem} Let $M$ be a compact Riemannian manifold. Then there exists a constant $C>0$, depending only on $M$, such that, for all $a,b,v\in C^\infty(M)$ with $a\geq 0$ and $\int_M a>0$,  if
$$
 \Delta v+ av=b,
$$
then
$$
\|v\|_{H^1}\leq 
C
\left\{
	\left(
		1+\frac{\|a\|_{L^2}}{\int_Ma}
	\right)
	\left(
			\|b\|_{L^2}+\frac{\|a\|_{L^2}}{\int_Ma}\left|\int_M b\right|
	\right)
	+\frac{|\int_M b|}{\int_M a}
\right\}.
$$
\end{lemma}

\begin{proof} Decompose $v=\ol{v}+v_0$ where 
\beq\label{cd}
\ol{v}=\frac{1}{|M|}\int_M v,
\eeq
so $v_0\in \X=\{f\in H^1: \int_M f=0\}$. We note that $\X$ is a Hilbert space with respect to the $H^1$ inner product. Substituting \eqref{cd} into the PDE for $v$ we find that
\beq\label{clda}
\Delta v_0 +av_0=b-a\ol{v},
\eeq
which, on integration over $M$ yields,
\beq\label{cladan}
\ol{v}=\frac{\int_Mb-\ip{a,v_0}_{L^2}}{\int_M a}.
\eeq
Substituting \eqref{cladan} back into \eqref{clda}, one finds that
\beq\label{claidane}
\Delta v_0 + a\left\{v_0 -\frac{\ip{a,v_0}_{L^2}}{\int_M a}\right\}
=\left\{b-\frac{\int_M b}{\int_M a}a\right\}.
\eeq
It is to \eqref{claidane} that we apply the Lax-Milgram Lemma.

Define the bilinear and linear forms
\bea
{\cal A}:\X\times\X\ra \R,\qquad {\cal A}(p,q)&=&\ip{\d p,\d q}_{L^2}
+\ip{ap,q}_{L^2}-\frac{1}{\int_M a}\ip{a,p}_{L^2}\ip{a,q}_{L^2}\\
{\cal B}:\X\ra\R,\qquad {\cal B}(q)&=&\ip{b-\frac{\int_M b}{\int_M a} a,q}_{L^2}.
\eea
Then, since $v_0$ satisfies \eqref{claidane}, for all $q\in \X$,
\beq
{\cal A}(v_0,q)={\cal B}(q).
\eeq
But ${\cal A}$ is {\em coercive} (with respect to the $H^1$ norm on $\X$) with coercivity constant $\lambda_1/(1+\lambda_1)$, where $\lambda_1>0$ is the smallest positive eigenvalue of $\Delta$, that is, for all $q\in\X$,
\beq\label{kb}
{\cal A}(q,q)\geq  \frac{\lambda_1}{1+\lambda_1}\|q\|_{H^1}^2.
\eeq
To see this, note that
\bea
{\cal A}(q,q)&=& \|\d q\|_{L^2}^2+\|\sqrt{a} q\|_{L^2}^2
-\frac{1}{\int_M a}\ip{\sqrt{a},\sqrt{a} q}_{L^2}^2 \nonumber \\
&\geq& 
\|\d q\|_{L^2}^2+\|\sqrt{a} q\|_{L^2}^2-\frac{1}{\int_M a}\|\sqrt{a}\|_{L^2}^2\|\sqrt{a}q\|_{L^2}^2\nonumber \\
&=& \|\d q\|_{L^2}
\eea
by Cauchy-Schwarz. But $\|\d q\|_{L^2}^2\geq \lambda_1\|q\|_{L^2}^2 $ by the Poincar\'e inequality, so
\beq
\|q\|_{H^1}^2=\|\d q\|_{L^2}^2+\|q\|_{L^2}^2\leq (1+\lambda_1^{-1})\|\d q\|_{L^2}^2,
\eeq
and \eqref{kb} immediately follows.

Hence, by the Lax-Milgram Lemma \cite[p.\ 78]{giltru}, 
\beq\label{kabe}
\|v_0\|_{H^1}\leq \frac{1+\lambda_1}{\lambda_1}\|{\cal B}\|_{\X^*}
\eeq
where
\bea
\|{\cal B}\|_{\X^*}&=&\sup_{q\in\X, \|q\|_{H^1}=1}|{\cal B}(q)|\nonumber \\
&=&\sup_{q\in\X, \|q\|_{H^1}=1}
\left|\ip{b-\frac{\int_Mb}{\int_Ma}a,q}_{L^2}\right|\nonumber \\
&\leq&\sup_{q\in\X, \|q\|_{H^1}=1}\|q\|_{L^2}\left\|
b-\frac{\int_Mb}{\int_Ma}a\right\|_{L^2}\nonumber \\
&\leq&
\left\|
b-\frac{\int_Mb}{\int_Ma}a\right\|_{L^2}\nonumber \\
\label{katbec}
&\leq&\|b\|_{L^2}+\frac{\|a\|_{L^2}}{\int_Ma}\left|\int_M b\right|.
\eea
But then, by \eqref{cladan},
\bea
|\ol{v}|&\leq&\frac{|\int_Mb|}{\int_Ma}+\frac{\|a\|_{L^2}}{\int_Ma}\|v_0\|_{L^2},
\eea
so
\bea
\|v\|_{H^1}&=&\|v_0+\ol{v}\|_{H^1}
\leq \|v_0\|_{H^1}+\sqrt{|M|}|\ol{v}|\nonumber \\
&\leq&\label{katebeck}
\left(1+\sqrt{|M|}\frac{\|a\|_{L^2}}{\int_Ma}\right)\|v_0\|_{H^1}+\sqrt{|M|}\frac{|\int_Mb|}{\int_Ma}.
\eea
Defining $C:=(1+\lambda_1^{-1})\max\{1,\sqrt{|M|}\}$, the claimed estimate now follows from inequalities \eqref{kabe}, \eqref{katbec} and
\eqref{katebeck}.
\end{proof}

We next apply Lemma \ref{LMlem} to the PDEs for $\dot\chi$,
$\dot u$, \eqref{gperp}, \eqref{hmvlin}, to obtain order $\eps$ bounds on
their $H^1$ norms. The last term in the inequality in Lemma \ref{LMlem}, namely $|\int_\Sigma b|/\int_\Sigma a$, looks fatal to this endeavour, since both $a$ and $b$ are of order $\eps$. It will turn out, however, that $\int_\Sigma b$ is actually of order $\eps^2$ due to subtle cancellations in the integral.

\begin{prop}\label{jodgas}
There exists a constant $C>0$, depending only on $(L,h)$ and $g_\Sigma$, such that
$$
\|\dot\chi\|_{H^1},\, \|\dot{u}\|_{H^1}\leq C\eps\|\dot{\wh\phi}\|_{L^2}.
$$
\end{prop}

\begin{proof} Throughout this proof, $c$ will denote a positive constant depending (at most) on $g_\Sigma$ and $(L,h)$. The value of $c$ may vary from line to line. Let $a=\eps|\wh\phi|^2e^u\geq 0$, as 
defined in \eqref{abdef}. Since $\wh\phi^{-1}(0)$ is finite, it is clear that $\int_\Sigma a>0$. By Lemmas \ref{lem1} and \ref{lem2} and the Mean Value Theorem,
\beq
e^u=1+U
\eeq
where $\|U\|_{C^0}\leq \|u\|_{C^0}e^{\|u\|_{C^0}}\leq c\eps$. Hence
\bea
\int_\Sigma a&=&\eps\int_\Sigma|\wh\phi|^2(1+U)\nonumber \\
&=&\eps+\eps\int_\Sigma|\wh\phi|^2U\nonumber \\
&\geq& \eps -\eps\|U\|_{C^0}\int_\Sigma|\wh\phi|^2\nonumber \\
&\geq& \eps-c\eps^2\geq c\eps.
 \eea 
Furthermore,
\bea
\|a\|_{L^2}^2&=&\eps^2\int_\Sigma |\wh\phi|^4e^{2u}\nonumber \\
&\leq& 2\eps^2\int_\Sigma |\wh\phi|^4(1+U^2)\nonumber \\
&\leq& 2\eps^2\alpha^2(1+\|U\|_{C^0}^2)\nonumber \\
&\leq& c\eps^2.
\eea

Consider now the function $b$ defined in \eqref{abdef}. Note that
\beq\label{fdef}
b=-\eps e^u k,\qquad k:=\left\{\begin{array}{cc}
h(i\wh\phi,\dot{\wh{\phi}}), & v=\dot\chi,\\
h(\wh\phi,\dot{\wh{\phi}}), & v=\dot u/2,\end{array}\right.
\eeq
and, crucially, in either case $\int_\Sigma k=0$ by
\eqref{hopf1} or \eqref{hopf2}.
Hence
\beq
\int_\Sigma b=\eps\int_\Sigma(1+U)k=\eps\int_\Sigma Uk,
\eeq
and so
\bea
\left|\int_\Sigma b\right|&\leq& \eps\|U\|_{C^0}\int_\Sigma |k|\nonumber \\
&\leq& \eps\|U\|_{C^0}\int_\Sigma |\wh\phi||\dot{\wh{\phi}}|\nonumber \\
&\leq& \eps\|U\|_{C^0}\|\wh\phi\|_{L^2}\|\dot{\wh{\phi}}\|_{L^2}\nonumber \\
&=&\eps\|U\|_{C^0}\|\dot{\wh{\phi}}\|_{L^2}\nonumber \\
&\leq& c\eps^2\|\dot{\wh{\phi}}\|_{L^2}.
\eea
Furthermore
\bea
\|b\|_{L^2}&=&\eps\|e^uf\|_{L^2}\nonumber \\
&\leq&\eps e^{\|u\|_{C^0}}\|\wh\phi\|_{C^0}\|\dot{\wh\phi}\|_{L^2}\nonumber \\
&\leq& c\eps\|\dot{\wh\phi}\|_{L^2}
\eea
by Lemma \ref{lem1}. Hence
\beq
\frac{\|a\|_{L^2}}{\int_\Sigma a}\leq c,\qquad\mbox{and}\qquad
\frac{|\int_\Sigma b|}{\int_\Sigma a}\leq c\eps\|\dot{\wh\phi}\|_{L^2}.
\eeq

Hence, by Lemma \ref{LMlem}, both $v=\dot\chi$ and $v=\dot u/2$ satisfy
\bea
\|v\|_{H^1}&\leq& c\left(\|b\|_{L^2}+|\int_\Sigma b|+\frac{|\int_\Sigma b|}{\int_\Sigma a}\right)\nonumber \\
&\leq& c\eps\|\dot{\wh\phi}\|_{L^2},
\eea
whence the claim follows.
\end{proof}

\begin{prop}\label{ps}
There exists a constant $C>0$, depending only on $(L,h)$ and $g_\Sigma$, such that
$$
\left| \|\dot\phi\|_{L^2}^2+\|\dot{A}\|_{L^2}^2-\eps\|\dot{\wh\phi}\|_{L^2}^2\right|\leq C\eps^2\|\dot{\wh\phi}\|_{L^2}^2.
$$
\end{prop}

\begin{proof}
Again, $c$ will denote a positive constant depending (at most) on $g_\Sigma$ and $(L,h)$, whose value may change from line to line. 
Differentiating \eqref{jg} and \eqref{joga} with respect to $t$, we see that
\bea
\|\dot\phi\|_{L^2}^2&=&
\eps\|\dot{\wh\phi}\|_{L^2}^2+\eps\int_\Sigma U|\dot{\wh\phi}|^2
+\eps\int_\Sigma\left\{h(\dot{\wh\phi},(\dot{u}+2i\dot\chi)\wh\phi) 
+\left(\frac14\dot u^2+\dot\chi^2\right)|\wh\phi|^2\right\}e^u\nonumber\\
\|\dot{A}\|_{L^2}^2&=&\frac14\|\d\dot{u}\|_{L^2}^2+\|\d\dot\chi\|_{L^2}^2,
\eea
where, as before, $e^u=:1+U$ with $\|U\|_{C^0}\leq c\eps$. Hence
\bea
\left| \|\dot\phi\|_{L^2}^2+\|\dot{A}\|_{L^2}^2-\eps\|\dot{\wh\phi}\|_{L^2}^2\right|
&\leq& \eps\|U\|_{C^0}\|\dot{\wh\phi}\|_{L^2}^2
+\eps\|\wh\phi\|_{C^0}e^{\|u\|_{C^0}}\|\dot{\wh\phi}\|_{L^2}(\|\dot{u}\|_{L^2}+2\|\dot\chi\|_{L^2})\nonumber \\
& &+\eps e^{\|u\|_{C^0}}\|\wh\phi\|_{C^0}^2\left(\frac14\|\dot u\|_{L^2}^2+\|\dot\chi\|_{L^2}^2\right)+\frac14\|\dot u\|_{H^1}^2
+\|\dot\chi\|_{H^1}^2\nonumber \\
&\leq& c\eps^2+\eps\alpha c\|\dot{\wh\phi}\|_{L^2}
3c\eps\|\dot{\wh\phi}\|_{L^2}+\eps c\alpha^23C^2\eps^2
\|\dot{\wh\phi}\|_{L^2}^2+\frac34c^2\eps^2
\|\dot{\wh\phi}\|_{L^2}^2\nonumber \\
&=&c\eps^2\|\dot{\wh\phi}\|_{L^2}^2,
\eea
where we have used Lemma \ref{lem1} to bound $e^{\|u\|_{C^0}}$ and Proposition \ref{jodgas}.

\end{proof}
 
Our main theorem immediately follows.

\begin{lproof}{of Theorem \ref{main}}
We must show that the symmetric bilinear form $\sigma_\eps:= g_\eps-g_0$ converges uniformly to zero, as $\eps\searrow 0$, on the unit tangent bundle of any fixed reference metric for $M_n$. 

Given any tangent vector $X\in T_{[(\phi,A)]}M_n$ we may choose a representative curve $(\phi(t),A(t))$ of the form \eqref{jg},
\eqref{joga} to compute
\beq
g_\eps(X,X)=\frac1\eps(\|\dot\phi\|_{L^2}^2+\|\dot{A}\|_{L^2}^2),
\eeq
and
\beq
g_0(X,X)=\|\dot{\wh\phi}\|_{L^2}^2.
\eeq
Proposition \ref{ps} implies that
\beq\label{ad}
|g_\eps(X,X)-g_0(X,X)|\leq C\eps g_0(X,X)
\eeq
where the constant $C$ is independent of $X$ and $\eps$. Hence,
for all unit length vectors ${X},{Y}$ tangent to $(M_n,  g_0)$,
\bea
|\sigma_\eps({X},{Y})|&=& 
\frac14|\sigma_\eps({X}+{Y},{X}+{Y})-
\sigma_\eps({X}-{Y},{X}-{Y})| \nonumber \\
&\leq&\frac14C\eps( g_0({X}+{Y},{X}+{Y})+
g_0({X}-{Y},{X}-{Y})) \nonumber \\
&=& C\eps.
\eea
 This establishes uniform convergence with respect to the reference metric $ g_0$.

\end{lproof}

\section{Convergence of the spectrum}\label{sec5}\news

Recall that the {\em spectrum} of a closed Riemannian manifold $(M,g)$ is the spectrum of its Laplace-Beltrami operator $\Delta_g=\delta\d$, arranged in non-decreasing order,
\beq
0=\lambda_0(g)<\lambda_1(g)\leq \lambda_2(g)\leq \lambda_3(g)\leq\cdots,
\eeq
a sequence of numbers that diverges to infinity. In the case of
$(M_n,g_\eps)$, this spectrum has a direct physical interpretation in terms of the quantum dynamics of $n$-vortices. The Hamiltonian operator governing this dynamics is (conjecturally) $H=\frac12\Delta_g$, where $g=\eps g_\eps$ is the (non-normalized) $L^2$ metric. Hence, the energy spectrum of $n$ quantum vortices moving on $\Sigma$ is
\beq
E_k(\eps)=\frac{\lambda_k(g_\eps)}{2\eps},\qquad k=0,1,2,\ldots.
\eeq
Note that, in local coordinates,
\beq
\Delta_{g_\eps}=-g^{ij}_\eps\left(\frac{\cd\:}{\cd x_i \cd x_j}
-\Gamma^{\eps,k}_{ij}\frac{\cd\: }{\cd x_k}\right),
\eeq
where $\Gamma^\eps$ denotes the Christoffel symbols of the metric $g_\eps$.
Since $\Gamma^\eps$ involves derivatives of $g_\eps$, the $C^0$ convergence proved in Theorem \ref{main} does {\em not} imply that
$\Delta_{g_\eps}$ converges to $\Delta_{g_0}$. Rather remarkably, we can still conclude, by work of Bando and Urakawa \cite{banura}, that the
{\em spectrum} of $\Delta_{g_\eps}$ converges uniformly to the
{\em spectrum} of $\Delta_{g_0}$, in the sense that
\beq
\frac{\lambda_k(g_\eps)}{\lambda_k(g_0)}\stackrel{\eps\ra 0}{\longrightarrow} 1
\eeq
uniformly in $k$. More precisely:

\begin{cor}\label{speccy}
Let $C>0$ be the constant, depending only on $(L,h)$ and $g_\Sigma$, whose existence is established in Proposition \ref{ps}. Then
 for all $\eps\in[0,1/C)$ and all
$k\geq 1$,
$$
\frac{(1-C\eps)^n}{(1+C\eps)^{n+1}}\leq\frac{\lambda_k(g_\eps)}{\lambda_k(g_0)}\leq\frac{(1+C\eps)^n}{(1-C\eps)^{n+1}}.
$$
\end{cor}

\begin{proof} Given a Riemannian metric $g$ on $M_n$, and a finite dimensional subspace $V$ of $C^\infty(M_n)$, let
\beq
\Lambda_g(V):=\sup\{\|\d f\|_{L^2,g}^2/\|f\|_{L^2,g}^2\: :\:
f\in V\less\{0\}\}.
\eeq
Then, by Proposition 2.1 of \cite{banura},
\beq
\lambda_k(g)=\inf\{\Lambda_g(V)\: :\: V< C^\infty(M_n),\, 
\dim V=k+1\},
\eeq
the infimum of $\Lambda_g$ over all $(k+1)$-dimensional subspaces of $C^\infty(M_n)$.

At a given point $p\in M_n$, let $E_1,\ldots, E_{2n}$ be an orthonormal frame for $(T_pM_n,g_0)$, $G_\eps$ be the matrix with
entries $g_\eps(E_i,E_j)$, $S_\eps$ the matrix with entries 
$\sigma_\eps(E_i,E_j)$ (where, as before, $\sigma_\eps=g_\eps-g_0$) and
$H_\eps$ be the inverse of $G_\eps$. Then the volume form defined by
$g_\eps$ is 
\beq
\vol_{g_\eps}=v_\eps \vol_{g_0},
\eeq
 where $v_\eps(p)=
(\det G_\eps)^{1/2}$. Now, by \eqref{ad}, every eigenvalue $\mu$ of $S_\eps$ satisfies $|\mu|\leq C\eps$, so the eigenvalues $1+\mu$ of
$G_\eps=\I_{2n}+S_\eps$ lie in $[1-C\eps, 1+C\eps]$. Hence, for all $\eps<1/C$,
\beq\label{abda}
(1-C\eps)^n\leq v_\eps(p)\leq (1+C\eps)^n.
\eeq
Furthermore, the eigenvalues $\nu=(1+\mu)^{-1}$ of $H_\eps$ lie in
$[(1+C\eps)^{-1},(1-C\eps)^{-1}]$. Now defining, for any given smooth function $f$, 
$v$ to be the column matrix with entries $\d f_p(E_i)$,
\beq
|\d f_p|_{g_\eps}^2=v^TH_\eps v,\qquad
|\d f_p|_{g_0}^2=v^Tv.
\eeq
Hence
\beq\label{abida}
\frac{|\d f_p|_{g_0}^2}{(1+C\eps)}\leq
\nu_{min} |\d f_p|_{g_0}^2\leq
|\d f_p|_{g_\eps}^2\leq \nu_{max} |\d f_p|_{g_0}^2
\leq \frac{|\d f_p|_{g_0}^2}{(1-C\eps)}.
\eeq

The bounds \eqref{abda} and \eqref{abida} are independent of 
$p\in M_n$, so hold globally. Hence, for all $f\in C^\infty(M_n)$,
\bea
(1-C\eps)^n\|f\|_{L^2, g_0}^2\leq &\|f\|_{L^2,g_\eps}^2& \leq
(1+C\eps)^n\|f\|_{L^2, g_0}^2, \nonumber \\
\frac{(1-C\eps)^n}{1+C\eps}\|\d f\|_{L^2, g_0}^2\leq &\|\d f\|_{L^2,g_\eps}^2& \leq
\frac{(1+C\eps)^n}{1-C\eps}\|\d f\|_{L^2, g_0}^2,
\eea
and so, for all $f\neq 0$,
\bea
\frac{\|\d f\|_{L^2,g_\eps}^2}{\|f\|_{L^2,g_\eps}^2}
&\leq&\frac{(1+C\eps)^n}{(1-C\eps)^{n+1}}\frac{\|\d f\|_{L^2,g_0}^2}{\|f\|_{L^2,g_0}^2}\nonumber \\
\frac{\|\d f\|_{L^2,g_0}^2}{\|f\|_{L^2,g_0}^2}
&\leq&\frac{(1+C\eps)^{n+1}}{(1-C\eps)^{n}}\frac{\|\d f\|_{L^2,g_\eps}^2}{\|f\|_{L^2,g_\eps}^2}.
\eea
Hence, for any $(k+1)$ dimensional subspace $V$ of $C^\infty(M_n)$,
\bea
\Lambda_{g_\eps}(V)&\leq& \frac{(1+C\eps)^n}{(1-C\eps)^{n+1}}\Lambda_{g_0}(V), \nonumber \\
\Lambda_{g_0}(V)&\leq& \frac{(1+C\eps)^{n+1}}{(1-C\eps)^{n}}\Lambda_{g_\eps}(V),
\eea
and so
\bea
\lambda_k(g_0)&\geq&\frac{(1-C\eps)^{n+1}}{(1+C\eps)^n}\lambda_k(g_\eps),
\nonumber \\
\lambda_k(g_\eps)&\geq&\frac{(1-C\eps)^n}{(1+C\eps)^{n+1}}\lambda_k(g_0),
\eea
which completes the proof.
\end{proof}

The spectrum of $g_0$ is known explicitly \cite{bergaumaz},
\beq
\lambda_k(g_0)=4k(n+k),\qquad d_k(g_0)=n(n+2k)\left(\begin{array}{c}n-1+k\\ k
\end{array}\right)^2,
\eeq
where we have changed our labelling convention so that $0=\lambda_0<\lambda_1<\lambda_2<\cdots$ and $d_k$ is the degeneracy of $\lambda_k$. In order to extract explicit bounds on $\lambda_k(g_\eps)$ from Corollary \ref{speccy}, we need an explicit upper bound on the constant $C$. Such a bound should be obtainable when $\Sigma$ is given the round metric, since explicit formulae for $\wh\phi$ are available in this case \cite{bapman}. A particularly important issue, if one is to apply Corollary \ref{speccy} to the study of the quantum statistical mechanics of vortices
\cite{man-qsm}, is the $n$ dependence of $C$.

\section*{Acknowledgements}
JMS wishes to thank Gerasim Kokarev for bringing reference \cite{banura} to his attention.

\bibliographystyle{../BIBLIOGRAPHY/stylefile.bst}
\bibliography{../BIBLIOGRAPHY/bibliography.bib}

\end{document}